\documentclass{article}

\usepackage[english]{babel}
\usepackage[letterpaper,top=2cm,bottom=2cm,left=3cm,right=3cm,marginparwidth=1.75cm]{geometry}
\usepackage{indentfirst}

\usepackage{graphicx}
\usepackage[colorlinks=true, allcolors=blue]{hyperref}
\usepackage{amssymb,amsmath,amsthm,amsfonts,xcolor,enumerate,hyperref,comment,longtable,cleveref,amscd}

\newtheorem{theorem}{Theorem}
\newtheorem{lemma}[theorem]{Lemma}
\newtheorem{proposition}[theorem]{Proposition}
\newtheorem{remark}[theorem]{Remark}

\newtheorem{definition}[theorem]{Definition}

\title{}
\author{}
\date{}

\begin{document}
\maketitle

\begin{Huge}
    \begin{center}
    (Transposed) Poisson algebra structures on null-filiform associative algebras
\end{center}
\end{Huge}

\begin{center}

 {\bf  Jobir Adashev\footnote{Institute of Mathematics, Academy of Sciences of Uzbekistan, Tashkent; adashevjq@mail.ru},
Xursanoy Berdalova\footnote{
Institute of Mathematics, Academy of Sciences of Uzbekistan, Tashkent; \  xursanoy.muminova55@gmail.com},
Feruza Toshtemirova\footnote{Chirchik State Pedagogical University, Chirchik, Uzbekistan; \ feruzaisrailova45@gmail.com}
}

\end{center}

\begin{abstract}
In this paper we investigate classifications of all (transposed) Poisson algebras of the associated associative null-filiform algebra.
\end{abstract}



\section{Introduction}
Poisson algebras originated in the study of Poisson geometry during the 1970s and have since appeared in a wide range of mathematical and physical disciplines, including Poisson manifolds, algebraic geometry, operads, quantization theory, quantum groups, and both classical and quantum mechanics. More recently, the notion of transposed Poisson algebras has been introduced in \cite{Bai}, providing a dual perspective on Poisson structures and leading to novel algebraic frameworks. This concept extends classical Poisson geometry and has found applications in various algebraic structures, including Novikov-Poisson algebras and $3$-Lie algebras \cite{bfk22}. In \cite{jawo, kk21, YYZ07} Poisson structures on canonical algebras and the finitary incidence algebra of an arbitrary poset over a commutative unital ring are described.  In papers \cite{conj, ch24} authors obtained a rich family of identities for transposed Poisson $n$-Lie algebras, and classified transposed Poisson $3$-Lie algebras of dimension $3$.

Research on transposed Poisson structures has gained significant traction in recent years. In \cite{bfk23}, a comprehensive algebraic and geometric classification of transposed Poisson algebras was provided, further expanding the understanding of their structural properties. Various works have explored transposed Poisson structures on different classes of Lie algebras. For instance, \cite{kk22, kk23, kkg23} examined such structures on Block Lie algebras, superalgebras, and Witt-type algebras, revealing intricate interactions between transposed Poisson and classical Poisson structures. The study of transposed Poisson structures on low-dimensional Lie algebras was advanced in \cite{ADSS}, where authors investigated the the transposed Poisson structures on low dimensional quasi-filiform lie algebras of maximum length.

Further progress has been made in the classification of transposed Poisson structures on specific algebraic frameworks. In \cite{KK7}, these structures were investigated in the context of upper triangular matrix Lie algebras. Additionally, the application of transposed Poisson structures to incidence algebras was explored in \cite{kkinc}, extending their relevance to combinatorial algebra. A different perspective was provided in \cite{FKL}, where the connection between transposed Poisson algebras and $\frac{1}{2}$-derivations of Lie algebras was analyzed, uncovering new insights into derivation-based algebraic properties.

The study of Poisson and transposed Poisson structures has also extended to infinite-dimensional and graded algebras. In \cite{KKhZ, KKhZ1}, transposed Poisson structures were examined in the context of not-finitely graded Witt-type algebras and Virasoro-type algebras, emphasizing their relevance in theoretical physics and representation theory. Meanwhile, the relationship between transposed Poisson structures and bialgebras was explored in \cite{lb23}. Also, using the connection between 12-derivations of Lie algebras and transposed Poisson algebras, descriptions of all transposed Poisson structures on some types of Lie algebras were obtained \cite{AAE, KKh, klv22, bl23, kms, yh21}.

The development of transposed Poisson structures continues to be an active area of research, with recent contributions exploring their connections to Jordan superalgebras \cite{fer23}, and Schrodinger algebras \cite{ytk}. Additionally, applications to classification problems in non-associative algebras have been addressed in \cite{k23}, further integrating transposed Poisson structures into the broader landscape of algebraic structures.

This paper provides a complete classification of all (transposed) Poisson algebra structures on null-filiform associative algebras. In Section 2, we introduce the necessary preliminary definitions and results that form the basis of our study. Using these foundations, we then describe all such structures in Section 3.

\section{PRELIMINARIES}

In this section, we introduce the relevant concepts and known results. Unless stated otherwise, all algebras considered here are over the field $\mathbb{C}$.

We first recall the definition of a Poisson algebra.

\begin{definition}
Let \(\mathfrak{L}\) be a vector space equipped with two bilinear operations
\[
\cdot, \; [-,-] : \mathfrak{L} \otimes \mathfrak{L} \to \mathfrak{L}.
\]
The triple \((\mathfrak{L}, \cdot, [-,-])\) is called a \textbf{Poisson algebra} if:
\begin{enumerate}
    \item \((\mathfrak{L}, \cdot)\) is a commutative associative algebra,
    \item \((\mathfrak{L}, [-,-])\) is a Lie algebra,
    \item The two operations satisfy the compatibility condition
    \begin{equation}\label{eq:LR}
    [x, y \cdot z] = [x, y] \cdot z + y \cdot [x, z], \quad \text{for all } x,y,z \in \mathfrak{L}.
    \end{equation}
\end{enumerate}
\end{definition}

Eq. (\ref{eq:LR}) is called the {\bf Leibniz
rule} since the adjoint operators of the Lie algebra are
derivations of the commutative associative algebra.

\begin{definition}
Let $\mathfrak{L}$ be a vector space equipped with two bilinear operations
$$
\cdot,\; [-,-] :\mathfrak{L}\otimes \mathfrak{L}\to \mathfrak{L}.$$
The triple $(\mathfrak{L},\cdot,[-,-])$ is called a \textbf{transposed Poisson algebra} if $(\mathfrak{L},\cdot)$ is a commutative associative algebra and $(\mathfrak{L},[-,-])$ is a Lie algebra which satisfies the following compatibility condition
\begin{equation}
2z\cdot [x,y]=[z\cdot x,y]+[x,z\cdot y].\label{eq:dualp}
\end{equation}
\end{definition}

Eq.~\eqref{eq:dualp} is called the {\bf transposed
Leibniz rule} because the roles played by the two binary operations in the Leibniz rule in a Poisson algebra are switched. Further, the resulting operation is rescaled by introducing a factor 2 on the left-hand side.

Transposed Poisson algebras were first introduced in the paper by Bai, Bai, Guo and Wu \cite{Bai}. A (transposed) Poisson algebra $\mathfrak{L}$ is called \textit{trivial}, if $\mathfrak{L}\cdot \mathfrak{L}=0$ or $[\mathfrak{L}, \mathfrak{L}]=0$.

The next result shows that the
compatibility relations of the Poisson algebra and those of the Transposed Poisson algebra are independent in the following sense.

\begin{proposition} [\cite{Bai}] Let $(\mathfrak{L},\cdot)$ be a commutative associative algebra and $(\mathfrak{L},[-,-])$ be a Lie algebra. Then $(\mathfrak{L},\cdot,[-,-])$ is both a
Poisson algebra and a Transposed Poisson algebra if and only
if

$$x\cdot [y,z]=[x\cdot y,z]=0.$$

\end{proposition}

For an algebra $\mathcal{A}$ of an arbitrary variety, we consider the series
\[
\mathcal{A}^1=\mathcal{A}, \quad \ \mathcal{A}^{i+1}=\sum\limits_{k=1}^{i}\mathcal{A}^k \mathcal{A}^{i+1-k}, \quad i\geq 1.
\]

We say that  an  algebra $\mathcal{A}$ is \emph{nilpotent} if $\mathcal{A}^{i}=0$ for some $i \in \mathbb{N}$. The smallest integer $i$ satisfying $\mathcal{A}^{i}=0$ is called the  \emph{index of nilpotency} or \emph{nilindex} of $\mathcal{A}$.

\begin{definition}
An $n$-dimensional algebra $\mathcal{A}$ is called null-filiform if $\dim \mathcal{A}^i=(n+ 1)-i,\ 1\leq i\leq n+1$.
\end{definition}

All null-filiform associative algebras were described in  \cite[Proposition 5.3]{MO}.

\begin{theorem}[\cite{MO}] An arbitrary $n$-dimensional null-filiform associative algebra is isomorphic to the algebra:
\[\mu_0^n : \quad e_i \cdot e_j= e_{i+j}, \quad 2\leq i+j\leq n,\]
where $\{e_1, e_2, \dots, e_n\}$ is a basis of the algebra $\mu_0^n$ and the omitted products vanish.
\end{theorem}



\begin{theorem}[\cite{aku}]  \label{3.1} A linear map $\varphi:\mu_0^n\to \mu_0^n $ is an automorphism of the algebra $\mu_0^n$ if and only if the map $\varphi$ has the
following form:
\[
\varphi(e_1)=\sum\limits_{i=1}^nA_ie_i,
\quad
\varphi(e_i)=\sum\limits_{j=i}^n \sum\limits_{k_1 +k_2 +... +k_i=j} A_{k_1}\cdot A_{k_2}\cdot ... \cdot A_{k_i}e_j, \quad 2\leq i\leq n,
\]
where $A_1\neq0$.
\end{theorem}

\section{(Transposed) Poisson algebras of null-filiform associative algebras}

This section describes all (transposed) Poisson algebra structures on null-filiform associative algebras.

\begin{theorem}\label{TP} Let $(\mu_0^n, \cdot, [-,-])$ be a transposed Poisson algebra structure defined on the associative  algebra $\mu_0^n$. Then the multiplication of $(\mu_0^n, \cdot,[-,-])$ has the following form:

$$\bf{TP}(\alpha_2,\dots,  \alpha_n):
\left\{\begin{array}{ll}
e_i\cdot e_j=e_{i+j}, &  2\leq i+j \leq n, \\[1mm]
[e_i,e_j]=(j-i)\sum\limits_{t=i+j-1}^{n}\alpha_{t-i-j+3}e_t, &  3\leq i+j \leq n+1.\\[1mm]
\end{array}\right.$$

\end{theorem}
\begin{proof} Let $(\mu_0^n, \cdot, [-,-])$ be a transposed Poisson algebra structure defined on the null-filiform associative algebra. To establish the table of multiplications Lie of the transposed structure of the Poisson algebra, we set

$$[e_1,e_2]=\sum\limits_{t=1}^{n}\alpha_{t}e_t. $$

Now, we consider the condition (\ref{eq:dualp}) for the triple $\{e_1, e_1, e_2\}$:
$$2e_1 \cdot [e_1,e_2]= [e_1\cdot e_1,e_2]+[e_1,e_1 \cdot e_2],$$
$$2e_1 \cdot \sum\limits_{t=1}^{n}\alpha_{t}e_t= [e_2,e_2]+[e_1,e_3].$$
From this, we derive the following product
$$[e_1,e_3]=2\sum\limits_{t=2}^{n}\alpha_{t-1}e_t.$$





By induction on $i$, we can show that
$$[e_1,e_i]=(i-1)\sum\limits_{t=i-1}^{n}\alpha_{t-i+2}e_t.$$
The equality holds for $i=2$. Assuming it holds for $i$, we prove the equality for $i+1$. We consider the identity (\ref{eq:dualp}) for the triple $\{e_1, e_1, e_i\}$:
$$2e_1 \cdot [e_1,e_{i}]= [e_1\cdot e_1,e_i]+[e_1,e_1 \cdot e_i],$$
$$2e_1 \cdot (i-1)\sum\limits_{t=i-1}^{n}\alpha_{t-i+2}e_t= [e_2,e_i]+[e_1,e_{i+1}].$$
Thus we obtain
\begin{equation}\label{eqstep1}[e_2,e_i]+[e_1,e_{i+1}]=2(i-1)\sum\limits_{t=i}^{n}\alpha_{t-i+1}e_t.\end{equation}
Similarly, considering (\ref{eq:dualp}) for the triple $\{e_{i-1}, e_1, e_2\}$:
$$2e_{i-1} \cdot [e_1,e_{2}]= [e_{i-1}\cdot e_1,e_2]+[e_1,e_{i-1} \cdot e_2],$$
$$2e_{i-1} \cdot \sum\limits_{t=1}^{n}\alpha_{t}e_t= [e_i,e_2]+[e_1,e_{i+1}],$$
we have
\begin{equation}\label{eqstep2}[e_i,e_2]+[e_1,e_{i+1}]=2\sum\limits_{t=i}^{n}\alpha_{t-i+1}e_t.\end{equation}
From equations (\ref{eqstep1}) and (\ref{eqstep2}), we deduce:
$$[e_1,e_{i+1}]=i\sum\limits_{t=i}^{n}\alpha_{t-i+1}e_t, \ [e_i,e_2]=(2-i)\sum\limits_{t=i}^{n}\alpha_{t-i+1}e_t.$$

Applying induction and the identity (\ref{eq:dualp}) for $2\leq i \leq  n$, we establish:
\begin{equation}\label{eq5}
[e_i,e_j]=(j-i)\sum\limits_{t=i+j-2}^{n}\alpha_{t-i-j+3}e_t.
\end{equation}

We can write
$$2e_{i} \cdot [e_1,e_{j}]= [e_{i}\cdot e_1,e_j]+[e_1,e_{i} \cdot e_j],$$
$$2e_{i} \cdot (j-1)\sum\limits_{t=j-1}^{n}\alpha_{t-j+2}e_t= [e_{i+1},e_j]+[e_1,e_{i+j}],$$
$$2(j-1)\sum\limits_{t=i+j-1}^{n}\alpha_{t-i-j+2}e_t= [e_{i+1},e_j]+(i+j-1)\sum\limits_{t=i+j-1}^{n}\alpha_{t-i-j+2}e_t.$$
Thus, we derive
$$[e_{i+1},e_j]=(j-i-1)\sum\limits_{t=i+j-1}^{n}\alpha_{t-i-j+2}e_t.$$

Next, applying the Jacobi identity to $\{e_1, e_3, e_n\}$, we get:
$$0=[[e_1,e_3],e_{n}]+[[e_3,e_{n}],e_1]+[[e_{n},e_1],e_3]$$
$$=[2\sum\limits_{t=2}^{n}\alpha_{t-1}e_t,e_{n}]+
[(1-n)\sum\limits_{t=n-1}^{n}\alpha_{t-n+2}e_t,e_3]$$
$$=2(n-2)\alpha_1
\sum\limits_{k=n}^{n}\alpha_{k-n+1}e_k+
(1-n)(4-n)\alpha_1
\sum\limits_{k=n}^{n}\alpha_{k-n+1}e_k$$
$$=(2(n-2)+(1-n)(4-n))\alpha_1^2e_n=(n^2-3n)\alpha_1^2e_n.$$
From here, we conclude that $\alpha_1=0$ for $n>3$.

Now, we consider two and three-dimensional cases separately.

For the case $n=2$, we analyze the identity (\ref{eq:dualp}) for the triple $\{e_1, e_1, e_2\}$:
$$0=[e_1\cdot e_1,e_2]+[e_1,e_1 \cdot e_2]=2e_1 \cdot [e_1,e_2]=2e_1\cdot (\alpha_{1}e_1+\alpha_{2}e_2)=2\alpha_1e_2.$$
From this, we have $\alpha_1=0.$

As for the case $n=3$, we analyze the identity (\ref{eq:dualp}) for the triple $\{e_2, e_1, e_2\}$:
$$[e_3,e_2]=[e_2\cdot e_1,e_2]+[e_1,e_2 \cdot e_2]=2e_2 \cdot [e_1,e_2]=2e_2\cdot (\alpha_{1}e_1+\alpha_{2}e_2+\alpha_{3}e_3)=2\alpha_1e_3.$$
On the other hand, from equality (\ref{eq5}), we have $[e_2,e_3]=\alpha_1e_3.$ Thus, we obtain $\alpha_1=0.$

Hence, we obtain the transposed Poisson algebras ${\bf TP}(\alpha_2,\dots, \alpha_n)$ given by the following multiplications:

$$[e_i,e_j]=(j-i)\sum\limits_{t=i+j-1}^{n}\alpha_{t-i-j+3}e_t, \quad 3\leq i+j \leq n+1.$$

\end{proof}

The following theorem establishes a necessary and sufficient condition for two algebras in the family ${\bf TP}(\alpha_2,\dots, \alpha_n)$ to be isomorphic.

\begin{theorem} Let ${\bf TP}(\alpha_2,\dots, \alpha_n)$ and ${\bf TP}'(\alpha_2', \dots,   \alpha_n')$ be isomorphic algebras. Then there exists an automorphism $\varphi$ between these algebras such that the following relation holds for $2\leq t\leq n$:

\begin{equation}\label{eq1}
    \sum\limits_{i=2}^t\sum\limits_{k_1  +\cdots+k_i=t} A_{k_1}... A_{k_i}\alpha_i' =\sum\limits_{j=2}^{t}\sum\limits_{i=1}^{t-j+1}  \sum\limits_{k_1+k_2=t-i-j+3} (t-2i-j+3) A_i  A_{k_1} A_{k_2} \alpha_{j}.
\end{equation}
Here, $\{e_1, e_2, \dots, e_n\}$ is a basis of the algebra ${\bf TP}(\alpha_2,\dots  \alpha_n)$.
\end{theorem}
\begin{proof}
Using the automorphism of the algebra $\mu_0^n$ from Theorem \ref{3.1}, we introduce the following notations:

\[
e_1'=\varphi(e_1), \quad
e_i'=\varphi(e_i),\quad 2\leq i\leq n.
\]
Thus, we consider
 $$\begin{array}{lcl}
    [e_1',e_2'] &= &\sum\limits_{i=2}^{n}\alpha_{i}'e_i'=\sum\limits_{i=2}^n \alpha_i'\sum\limits_{j=i}^n\sum\limits_{k_1 +... +k_i=j} A_{k_1}...A_{k_i}e_j \\[1mm]
    &=& \sum\limits_{i=2}^n \sum\limits_{j=i}^n \sum\limits_{k_1+\dots +k_i=j}\alpha_i' A_{k_1} ...A_{k_i}e_j=\sum\limits_{t=2}^n \sum\limits_{i=2}^t\sum\limits_{k_1  +\cdots+k_i=t} \alpha_i' A_{k_1} ...A_{k_i}e_t.  \\
        \end{array}$$

On the other hand, we have
 $$\begin{array}{lcl}
 [e_1',e_2']&=&[\sum\limits_{i=1}^n A_i e_i, \sum\limits_{j=2}^n\sum\limits_{k_1+k_2=j} A_{k_1} A_{k_2} e_j]=\sum\limits_{i=1}^n \sum\limits_{j=2}^n  \sum\limits_{k_1+k_2=j} A_iA_{k_1}A_{k_2} [e_i,e_j]\\[1mm]
 &=&\sum\limits_{i=1}^n \sum\limits_{j=2}^n \sum\limits_{k_1+k_2=j} A_i A_{k_1} A_{k_2}(j-i) \sum\limits_{t=i+j-1}^n \alpha_{t-i-j+3} e_t\\[1mm]
&=&\sum\limits_{i=1}^n \sum\limits_{j=2}^n \sum\limits_{t=i+j-1}^n \sum\limits_{k_1+k_2=j}(j-i) A_i A_{k_1} A_{k_2} \alpha_{t-i-j+3} e_t\\[1mm]
&=&\sum\limits_{i=1}^n \sum\limits_{t=i+1}^n \sum\limits_{j=2}^{t-i+1}\sum\limits_{k_1+k_2=j} (j-i)A_i A_{k_1} A_{k_2} \alpha_{t-i-j+3} e_t\\[1mm]
&=&\sum\limits_{t=2}^n\sum\limits_{i=1}^{t-1} \sum\limits_{j=2}^{t-i+1} \sum\limits_{k_1+k_2=j} (j-i)A_i A_{k_1}  A_{k_2} \alpha_{t-i-j+3} e_t.\end{array}$$

Comparing the coefficients of the obtained expressions for the basis elements for $2\leq t\leq n$, we get the following restrictions:
$$\sum\limits_{i=2}^t\sum\limits_{k_1  +\cdots+k_i=t} A_{k_1}... A_{k_i}\alpha_i'=\sum\limits_{i=1}^{t-1} \sum\limits_{j=2}^{t-i+1}  \sum\limits_{k_1+k_2=j} (j-i)A_i A_{k_1} A_{k_2} \alpha_{t-i-j+3}.$$

Now we transform the right-hand side of the equation. Firstly, we substitute $j$. Setting $m:=t-i-j+3$, we obtain $2\leq m \leq t-i+1$, so we replace $j$ with $t-i-m+3$. This yields:
$$\sum\limits_{i=1}^{t-1} \sum\limits_{j=2}^{t-i+1}  \sum\limits_{k_1+k_2=j} (j-i)A_iA_{k_1} A_{k_2} \alpha_{t-i-j+3}=\sum\limits_{i=1}^{t-1} \sum\limits_{m=2}^{t-i+1} \sum\limits_{k_1+k_2=t-i-m+3} (t-2i-m+3) A_i A_{k_1} A_{k_2} \alpha_{m}.$$

Now, we exchange the order of the first and second summations:
$$\sum\limits_{i=1}^{t-1} \sum\limits_{m=2}^{t-i+1}  \sum\limits_{k_1+k_2=t-i-m+3}(t-2i-m+3) A_i A_{k_1} A_{k_2} \alpha_{m}=\sum\limits_{m=2}^{t} \sum\limits_{i=1}^{t-m+1} \sum\limits_{k_1+k_2=t-i-m+3} (t-2i-m+3) A_i A_{k_1} A_{k_2} \alpha_{m}.$$

Finally, replacing $m$ with $j$, we get the following equality for $2\leq t\leq n$:
$$\sum\limits_{i=2}^t\sum\limits_{k_1  +\cdots+k_i=t} A_{k_1}...A_{k_i}\alpha_i'=\sum\limits_{j=2}^{t} \sum\limits_{i=1}^{t-j+1} \sum\limits_{k_1+k_2=t-i-j+3}(t-2i-j+3) A_iA_{k_1}A_{k_2} \alpha_{j}.$$

\end{proof}

\begin{lemma} Let ${\bf TP}(\alpha_2,\dots,  \alpha_n)$  be a transposed Poisson algebra structure defined on $\mu_0^n$. If $\alpha_i=0,$ for all $2\leq i< s$ for some $s$ $(2\leq s\leq n),$ then the parameter $\alpha_{s}$ is zero invariant. Moreover, there exists $A\in\mathbb{C}$ such that the following relation holds:
$$\alpha_{s}'=\frac{\alpha_{s}}{A^{s-3}}.$$
\end{lemma}

\begin{proof} Let ${\bf TP}(\alpha_2,\dots, \alpha_n)$ be a transposed Poisson algebra structure on commutative algebra $\mu_0^n$, and consider a general change of basis. Then, for some natural number $s$, we have the following restriction (\ref{eq1}):
$$ \sum\limits_{i=2}^s\sum\limits_{k_1  +\cdots+k_i=s} A_{k_1} ... A_{k_i}\alpha_i'=\sum\limits_{j=2}^{s} \sum\limits_{i=1}^{s-j+1}  \sum\limits_{k_1+k_2=s-i-j+3} (s-2i-j+3)A_i A_{k_1} A_{k_2} \alpha_{j}.$$

Now, considering that $\alpha_i=0$ for $2\leq i< s$, this reduces to
$$\sum\limits_{k_1  +\cdots+k_s=s} A_{k_1} ... A_{k_s}\alpha_s'=\sum\limits_{i=1}^{s-s+1} \sum\limits_{k_1+k_2=s-i-s+3} (s-2i-s+3)A_i A_{k_1} A_{k_2} \alpha_{s}$$
which simplifies further to
$$ \sum\limits_{k_1  +... +k_s=s} A_{k_1} ... A_{k_s}\alpha_s'= \sum\limits_{k_1+k_2=2}A_1 A_{k_1} A_{k_2} \alpha_{s}.$$
From this equality, we obtain the relation
$$\alpha_{s}'=\frac{\alpha_{s}}{A_1^{s-3}}.$$
\end{proof}

\begin{lemma}\label{lem2} Let $\mathfrak{L}$ be the algebra ${\bf TP}(\alpha_2,\dots, \alpha_n)$. If $\alpha_2\neq0,$ then it is
isomorphic to the algebra ${\bf TP}(1, 0, \dots, 0)$.
\end{lemma}

\begin{proof} Let $\alpha_2\neq0$ and consider the following change of basis:
$$e_i'=\frac{1}{\alpha_2^i}e_i, \ \ 1\leq i\leq n.$$
Applying this transformation to the product $[e_1,e_2]=\sum\limits_{t=2}^{n}\alpha_{t}e_t$, we can assume that $\alpha_2=1.$

Next, consider another change of basis for $e_1:$ $$e_1'=e_1+Ae_2.$$  From the relations $e_ie_j=e_{i+j}$ for $2\leq i+j\leq n,$ we obtain
$$e_i'=\sum\limits_{t=0}^{i}\left(\begin{array}{cc}
     i  \\
     t
\end{array}\right)A^te_{i+t}, \ \ 1\leq i\leq n,$$
where we assume that $e_t=0$ for $t>n$.
Setting $A=-\frac{\alpha_3}{2}$ and considering $[e_1,e_2]=e_2+\sum\limits_{t=3}^{n}\alpha_{t}e_t$ we conclude that $\alpha_3=0.$

Now, we prove by induction that it is possible to set $\alpha_j=0,$ for all $3\leq j\leq n$. The base case $j = 3$ is already established. Assuming the claim holds for some $j$, we show it also holds for $j + 1$.  Consider the change of basis: $$e_1'=e_1+Ae_j.$$  Using $e_ie_j=e_{i+j}$ for $2\leq i+j\leq n,$ we derive
$$e_i'=\sum\limits_{t=0}^{i}\left(\begin{array}{cc}
     i  \\
     t
\end{array}\right)A^te_{i+t(j-1)}, \ \ 1\leq i\leq n,$$
where we assume that $e_t=0$ for $t>n$ in this sum. Setting $A=-\frac{\alpha_{j+1}}{j}$ and considering $[e_1,e_2]=e_2+\sum\limits_{t=j+1}^{n}\alpha_{t}e_t$ we conclude that $\alpha_{j+1}=0.$

By induction, we have $\alpha_j=0,$ for all $3\leq j\leq n$, proving that the algebra $\mathfrak{L}$ is isomorphic to ${\bf TP}(1, 0, \dots, 0)$, with the following multimplication rules:
$$\left\{\begin{array}{ll}
e_ie_j=e_{i+j}, &  2\leq i+j \leq n, \\[1mm]
[e_i,e_j]=(j-i)e_{i+j-1}, &  3\leq i+j \leq n+1.\\[1mm]
\end{array}\right.$$
\end{proof}

\begin{lemma}\label{lem3} Let $\mathfrak{L}$ be the algebra ${\bf TP}(\alpha_2, \dots, \alpha_n)$. If $\alpha_2=0,\ \alpha_3\neq0,$ then $\mathfrak{L}$ is
isomorphic to the algebra ${\bf TP}(0,\alpha, 0, \dots, 0)$.
\end{lemma}

\begin{proof} Let $\alpha_3\neq0$, and consider the following change of basis: $$e_1'=e_1+Ae_2.$$
From the products $e_ie_j=e_{i+j}$ for $2\leq i+j\leq n,$ we obtain
$$e_i'=\sum\limits_{t=0}^{i}\left(\begin{array}{cc}
     i  \\
     t
\end{array}\right)A^te_{i+t}, \ \ 1\leq i\leq n,$$
where we assume that $e_t=0$ for $t>n$.
Setting $A=-\frac{\alpha_4}{\alpha_3}$, and considering the product $[e_1,e_2]=\alpha_3 e_3+\sum\limits_{t=4}^{n}\alpha_{t}e_t$, we conclude that $\alpha_4=0.$

Now, we prove by induction that we can eliminate $\alpha_j$ for all $4\leq j\leq n$. The base case $j = 4$ follows from the above step. Assuming that $\alpha_j=0$ holds for some $j\geq 4$, we prove it for $j + 1$. Consider the change of basis: $$e_1'=e_1+Ae_{j-1}.$$  From the products $e_ie_j=e_{i+j}$ for $2\leq i+j\leq n,$ we derive
$$e_i'=\sum\limits_{t=0}^{i}\left(\begin{array}{cc}
     i  \\
     t
\end{array}\right)A^te_{i+t(j-2)}, \ \ 1\leq i\leq n,$$
where we again assume that $e_t=0$ for $t>n$. Setting $A=-\frac{\alpha_{j+1}}{(j-2)\alpha_3}$, and considering the product $[e_1,e_2]=e_2+\sum\limits_{t=j+1}^{n}\alpha_{t}e_t$ we conclude that $\alpha_{j+1}=0.$

Thus, we have shown that $\alpha_j=0,$ for $4\leq j\leq n$, and the algebra $\mathfrak{L}$ is isomorphic to the algebra ${\bf TP}(0, \alpha, 0, \dots, 0)$:
$$\left\{\begin{array}{ll}
e_ie_j=e_{i+j}, &  2\leq i+j \leq n, \\[1mm]
[e_i,e_j]=(j-i)\alpha_3 e_{i+j}, &  3\leq i+j \leq n.\\[1mm]
\end{array}\right.$$

\end{proof}

\begin{lemma}\label{lem4} Let $\mathfrak{L}$ be the algebra ${\bf TP}(\alpha_2, \dots, \alpha_n)$. If $\alpha_2=\alpha_3=0,\ \alpha_4\neq0,$ then it is
isomorphic to the algebra ${\bf TP}(0,0, 1,\alpha,0, \dots, 0)$.
\end{lemma}

\begin{proof} From the equality \eqref{eq1}, we obtain the following equality in general substitution.

$$\alpha_4'=\frac{\alpha_4}{A_1}, \quad \alpha_5'=\frac{\alpha_5}{A_1^2}.$$

Since $\alpha_4\neq0$, we consider the following change of basis:
$$e_i'=\alpha_4^ie_i, \ \ 1\leq i\leq n.$$
Considering the product $[e_1,e_2]=\sum\limits_{t=4}^{n}\alpha_{t}e_t$ we can assume that $\alpha_4=1.$

Now, consider the change of basis: $e_1'=e_1+Ae_3.$  From the products $e_ie_j=e_{i+j}$ for $2\leq i+j\leq n,$ we derive
$$e_i'=\sum\limits_{t=0}^{i}\left(\begin{array}{cc}
     i  \\
     t
\end{array}\right)A^te_{i+2t}, \ \ 1\leq i\leq n,$$
where we assume that $e_t=0$ for $t>n$ in this sum.
Setting $A=-\alpha_6$, and considering the product $[e_1,e_2]=e_4+\sum\limits_{t=5}^{n}\alpha_{t}e_t$ we obtain that $\alpha_6=0.$

Now we prove by induction that it is possible to set $\alpha_j=0,$ for $6\leq j\leq n$. The base case $j = 6$ holds by the above argument. Now, assuming it holds for some $j$, we prove it for $j + 1$. Consider basis change: $e_1'=e_1+Ae_{j-2}.$  From the products $e_ie_j=e_{i+j}$ for $2\leq i+j\leq n,$ we derive
$$e_i'=\sum\limits_{t=0}^{i}\left(\begin{array}{cc}
     i  \\
     t
\end{array}\right)A^te_{i+t(j-3)}, \ \ 1\leq i\leq n,$$
where we assume that $e_t=0$ for $t>n$ in this sum. Setting $A=-\frac{\alpha_{j+1}}{j-4}$, and considering the product $[e_1,e_2]=e_2+\sum\limits_{t=j+1}^{n}\alpha_{t}e_t$ we obtain that $\alpha_{j+1}=0.$

Thus, we have shown that $\alpha_j=0,$ for $6\leq j\leq n$, and the algebra $\mathfrak{L}$ is isomorphic to the algebra ${\bf TP}(0,0,1, \alpha,0, \dots, 0)$:
$$\left\{\begin{array}{ll}
e_ie_j=e_{i+j}, &  2\leq i+j \leq n, \\[1mm]
[e_i,e_j]=(j-i)(e_{i+j+1}+\alpha_5e_{i+j+2}), &  3\leq i+j \leq n-1.
\end{array}\right.$$

\end{proof}

\begin{lemma}\label{lems} Let $\mathfrak{L}$ be the algebra ${\bf TP}(\alpha_2, \dots, \alpha_n)$. If $\alpha_2=\dots=\alpha_{s-1}=0,\ \alpha_s\neq0, \ s\geq 5$ then it is
isomorphic to the algebra ${\bf TP}(0,\dots, 0, 1_s,0,\dots,0,\alpha_{2s-3},0, \dots, 0)$.
\end{lemma}

\begin{proof}
Let $\alpha_s\neq0$ and consider the change of basis as follows:
$$e_i'=\alpha_s^{\frac{i}{s-3}}e_i, \ \ 1\leq i\leq n.$$
Considering the product $[e_1,e_2]=\sum\limits_{t=s}^{n}\alpha_{t}e_t$ we can assume that $\alpha_s=1.$

Now consider the change of basis as follows: $e_1'=e_1+Ae_2.$  From the products $e_ie_j=e_{i+j}$ for $2\leq i+j\leq n,$ we derive
$$e_i'=\sum\limits_{t=0}^{i}\left(\begin{array}{cc}
     i  \\
     t
\end{array}\right)A^te_{i+t}, \ \ 1\leq i\leq n,$$
where we assume that $e_t=0$ for $t>n$ in this sum.
Setting $A=\frac{\alpha_{s+1}}{s-4}$, and considering the product $[e_1,e_2]=e_s+\sum\limits_{t=s+1}^{n}\alpha_{t}e_t$ we can derive that $\alpha_{s+1}=0.$

Now we prove by induction using successive changes of the basis elements that it is possible to make $\alpha_j=0,$ for $s+1\leq j\leq 2s-5$. If $j = s+1$, the relationship is fulfilled according to the above equalities. Now, that it holds for some $j$, we prove it for $j + 1$. We consider the change of the basis element $e_1$ as $e_1'=e_1+Ae_{j-s+2}.$  From the products  $e_ie_j=e_{i+j}$ for $2\leq i+j\leq n,$ we derive
$$e_i'=\sum\limits_{t=0}^{i}\left(\begin{array}{cc}
     i  \\
     t
\end{array}\right)A^te_{i+tj}, \ \ 1\leq i\leq n,$$
where we assume that $e_t=0$ for $t>n$ in this sum. Setting $A=\frac{\alpha_{j+1}}{2s-4-j}$, and considering the product $[e_1,e_2]=e_s+\sum\limits_{t=j+1}^{n}\alpha_{t}e_t$ we can derive that $\alpha_{j+1}=0.$
Then we have
$$[e_1,e_2]=e_s+\sum\limits_{t=2s-3}^{n}\alpha_{t}e_t.$$

From the equality \eqref{eq1}, we obtain the following equality in general substitution.

$$\alpha_{2s-3}'=\alpha_{2s-3}.$$

Now, consider the change of basis elements as follows: $e_1'=e_1+Ae_{s-1}.$  From the products $e_ie_j=e_{i+j}$ for $2\leq i+j\leq n,$ we derive
$$e_i'=\sum\limits_{t=0}^{i}\left(\begin{array}{cc}
     i  \\
     t
\end{array}\right)A^te_{i+t(s-2)}, \ \ 1\leq i\leq n,$$
where we assume that $e_t=0$ for $t>n$ in this sum.
Setting $A=-\alpha_{2s-2}$, and considering the product $[e_1,e_2]=e_s+\sum\limits_{t=s+1}^{n}\alpha_{t}e_t$ we can derive that $\alpha_{2s-2}=0.$

Now we prove by induction that it is possible to set $\alpha_j=0,$ for $2s-2\leq j\leq n$. If $j = 2s-2$, the relationship is fulfilled according to the above equalities. Now, that it holds for $j$, we prove it for $j + 1$.  Consider the change of basis element $e_1$ as $e_1'=e_1+Ae_{j-s+2}.$  From the products $e_ie_j=e_{i+j}$ for $2\leq i+j\leq n,$ we derive
$$e_i'=\sum\limits_{t=0}^{i}\left(\begin{array}{cc}
     i  \\
     t
\end{array}\right)A^te_{i+t(j-s+1)}, \ \ 1\leq i\leq n,$$
where we assume that $e_t=0$ for $t>n$ in this sum. Setting $A=-\frac{\alpha_{j+1}}{2s-4-j}$, and considering the product $[e_1,e_2]=e_2+\sum\limits_{t=j+1}^{n}\alpha_{t}e_t$ we can derive that $\alpha_{j+1}=0.$

Thus, we have shown that $\alpha_j=0,$ for $2s-2\leq j\leq n$, and the algebra $\mathfrak{L}$ is isomorphic to the algebra ${\bf TP}(0,\dots, 0, 1_s,0,\dots,0,\alpha_{2s-3},0, \dots, 0)$:
$$\left\{\begin{array}{ll}
e_ie_j=e_{i+j}, &  2\leq i+j \leq n, \\[1mm]
[e_i,e_j]=(j-i)(e_{s+i+j-3}+\alpha_{2s-3}e_{2s+i+j-6}), &  3\leq i+j \leq n-s+3.\\[1mm]
\end{array}\right.$$

\end{proof}

\begin{theorem} Let $(\mu_0^n, \cdot, [-,-])$ be a transposed Poisson algebra structure defined on the associative  algebra $\mu_0^n$ for $n\geq 5$. Then this algebra is isomorphic to one of the following pairwise non-isomorphic algebras:
\[{\bf TP}(1,0,\dots,0), \ {\bf TP}(0,\alpha,0, \dots, 0), \ {\bf TP}(0,\dots, 0, 1_s,0,\dots,0,\alpha_{2s-3},0, \dots, 0), \ 4\leq s\leq n, \ \alpha\in\mathbb{C}.\]
\end{theorem}

\begin{proof} Let $(\mu_0^n, \cdot, [-,-])$ be a transposed Poisson algebra structure defined on the associative  algebra $\mu_0^n$. Then Theorem \ref{TP} implies that it is isomorphic to the algebra ${\bf TP}(\alpha_2,\dots, \alpha_n)$. Moreover,
\begin{itemize}
    \item If $\alpha_2\neq 0$ then from Lemma \ref{lem2} we have ${\bf TP}(1,0,\dots,0)$ algebra;
    \item If $\alpha_2=0, \ \alpha_3\neq 0$ then from Lemma \ref{lem3} we have ${\bf TP}(0,\alpha,0, \dots, 0)$ algebra, where $\alpha\neq0$;
    \item If $\alpha_2=\alpha_3=0, \ \alpha_4\neq 0$ then from Lemma \ref{lem4} we have ${\bf TP}(0,0, 1,\alpha,0, \dots, 0)$ algebra;
    \item If $\alpha_2=\dots=\alpha_{s-1}=0, \ \alpha_s\neq 0, \ s\geq5$ then from Lemma \ref{lems} we have the algebra $${\bf TP}(0,\dots, 0, 1_s,0,\dots,0,\alpha_{2s-3},0, \dots, 0), \ 5\leq s\leq n;$$
    \item If $\alpha_i=0, \ 2\leq i\leq n$ then we have ${\bf TP}(0,\dots,0)$ algebra.
\end{itemize} \end{proof}

In the following theorems, we consider 2,3 and 4-dimensional cases. In \cite{Bai} proved that any 2-dimensional complex transposed Poisson algebra on the 2-dimensional associative algebra whose product is given $e_1\cdot e_1=e_2$ is isomorphic to one of the following transposed Poisson algebras:
$${\bf TP}(0): \ e_1\cdot e_1=e_2; \quad {\bf TP}(1): \ e_1\cdot e_1=e_2, \ [e_1,  e_2]=e_2.$$

\begin{theorem} Let $(\mu_0^3, \cdot, [-,-])$ be a transposed Poisson algebra structure defined on the associative  algebra $\mu_0^3$. Then this algebra is isomorphic to one of the following pairwise non-isomorphic algebras:
\[{\bf TP}(1,0), \ {\bf TP}(0,\alpha), \ \alpha\in\mathbb{C}.\]
\end{theorem}

\begin{proof} Let $(\mu_0^3, \cdot, [-,-])$ be a transposed Poisson algebra structure defined on a null-filiform associative algebra. Then according to Theorem \ref{TP} it is isomorphic to the algebra ${\bf TP}(\alpha_2,\alpha_3)$. Applying a change of basis and using equality \eqref{eq1}, we get the following relations:
$$\alpha_2'=A_1\alpha_2, \ \alpha_3'=\frac{A_1\alpha_3+2A_2\alpha_2}{A_1}.$$

\begin{itemize}
    \item Let $\alpha_2\neq 0$. Then, by selecting
    $$A_1=\frac{1}{\alpha_2}, \ A_2=-\frac{\alpha_3}{2\alpha_2^2},$$
    we have the algebra ${\bf TP}(1,0)$;
    \item Let $\alpha_2=0$. Then we derive the algebra ${\bf TP}(0,\alpha).$
\end{itemize}

\end{proof}

\begin{theorem} Let $(\mu_0^4, \cdot, [-,-])$ be a transposed Poisson algebra structure defined on the associative  algebra $\mu_0^4$. Then this algebra is isomorphic to one of the following pairwise non-isomorphic algebras:
\[{\bf TP}(1,0,0), \ {\bf TP}(0,\alpha,0), \ {\bf TP}(0,0,1), \ \alpha\in\mathbb{C}.\]
\end{theorem}

\begin{proof} Let $(\mu_0^n, \cdot, [-,-])$ be a transposed Poisson algebra structure defined on the associative  algebra $\mu_0^n$. According to Theorem \ref{TP}, this algebra is isomorphic to the algebra ${\bf TP}(\alpha_2,\alpha_3,\alpha_4)$. Applying a change of basis and using the equality \eqref{eq1}, we get the following relations:
$$\alpha_2'=A_1\alpha_2, \ \alpha_3'=\frac{A_1\alpha_3+2A_2\alpha_2}{A_1}, \ \alpha_4'=\frac{A_1^2 \alpha_4+ A_1 A_2 \alpha_3+(3A_1A_3-2A_2^2)\alpha_2}{A_1^3}.$$

\begin{itemize}
    \item Let $\alpha_2\neq 0$. Then by selecting
    $$A_1=\frac{1}{\alpha_2}, \ A_2=-\frac{\alpha_3}{2\alpha_2^2}, \ A_3=\frac{\alpha_3^2-\alpha_2\alpha_4}{3\alpha_2^3},$$
    we obtain the algebra ${\bf TP}(1,0,0)$;
    \item Let $\alpha_2=0$ and $\alpha_3\neq 0.$ Choosing
    $$A_1=-\alpha_3, \ A_2=\alpha_4,$$
    we derive the algebra ${\bf TP}(0,\alpha,0),$ where $\alpha\neq0$;
    \item Let $\alpha_2=\alpha_3=0$ and $\alpha_4\neq 0$. Then, by setting $A_1=\alpha_1$, we obtain the algebra ${\bf TP}(0,0,1)$;
    \item If $\alpha_2=\alpha_{3}=\alpha_4=0,$ then the algebra simplifies to ${\bf TP}(0,0,0)$.
\end{itemize}

\end{proof}

\begin{remark} Thus, from the following table and the obtained results, it can be seen that the classifications of the transposed Poisson algebra structures defined on the associative algebra $\mu_0^n$ leads to the following pairwise non-isomorphic algebras:

\begin{itemize}
    \item Transposed Poisson algebras ${\bf TP}(0,\alpha,0, \dots, 0)$ and ${\bf TP}(0,\dots, 0, 1_s,0,\dots,0,\alpha_{2s-3},0, \dots, 0),$ for $4\leq s\leq n, \ \alpha\in\mathbb{C}$ are nilpotent algebras. Furthermore, the algebra ${\bf TP}(0,\dots,0)$  is a trivial transposed Poisson algebra;
    \item Now we consider the algebra ${\bf TP}(1,0,\dots,0)$. By performing a change of basis as follows:
$$x=e_1,  \ e_i'=e_{i+1}, 1\leq i\leq n-1,$$
we obtain the multiplication table of ${\bf TP}(1,0,\dots,0)$ with respect to the product $[-,-]$ in the basis $\{x,e_1,\dots,e_{n-1}\}$:
$$[e_i,e_j]=(j-i)e_{i+j}, \ 3\leq i+j \leq n$$
$$[x,e_i]=ie_{i}, \ 1\leq i\leq n-1.$$
Hence, the algebra ${\bf TP}(1,0,\dots,0)$ is a solvable transposed Poisson algebra with its nilradical spanned by $\{e_2,\dots,e_{n}\}$.\end{itemize}

\end{remark}

All Poisson algebra structures on null-filiform associative algebras were constructed in \cite{AFM}. Below we attempt to discuss the same problem with a different approach.

\begin{theorem} Let $(\mu_0^n, \cdot, [-,-])$ be a  Poisson algebra structure defined on the associative  algebra $\mu_0^n$. Then $(\mu_0^n, \cdot, [-,-])$ is a trivial Poisson algebra.
\end{theorem}
\begin{proof} Let $(\mu_0^n, \cdot, [-,-])$ be a Poisson algebra structure defined on a null-filiform associative algebra. To establish the table of multiplications for the operation $[-,-]$ in this Poisson algebra structure, we consider the following computation for $1\leq i\leq n-1$:
$$[e_1,e_{i+1}]=[e_1,e_1 \cdot e_i]= [e_1,e_1]\cdot e_i+e_1\cdot [e_1,e_i]=e_1\cdot [e_1,e_i].$$
From this we get $[e_1,e_i]=0, \ 2\leq i\leq n.$

Next, considering the following equalities
$$[e_i,e_{2}]=[e_i,e_1 \cdot e_1]= [e_i,e_1]\cdot e_1+e_1\cdot [e_i,e_1]=0,\ 3\leq i \leq n-1,$$
$$[e_i,e_{j}]=[e_i,e_1 \cdot e_{j-1}]= [e_i,e_1]\cdot e_{j-1}+e_1\cdot [e_i,e_{j-1}]=0,\ 3\leq i,j\leq n,$$
we obtain
$$[e_i,e_j]=0, \quad 1\leq i,j\leq n.$$
\end{proof}

\end{document}